\documentclass[letterpaper, 10 pt, conference]{ieeeconf}  

\IEEEoverridecommandlockouts                              

\usepackage{amsmath} 
\usepackage{amssymb}  
\usepackage{bm}
\usepackage{bbm}
\usepackage{xcolor}
\usepackage{acronym}
\usepackage{graphicx}
\usepackage{comment}
\usepackage{subcaption}
\usepackage{dsfont}
\usepackage{multicol}

\newacro{pdf}[PDF]{probability density function}
\newacro{kld}[KLD]{Kullback-Leibler divergence}
\newacro{cdf}[CDF]{cummulative distribution function}
\newacro{vgub}[VGUB]{variational Gaussian upper bound}
\newacro{scvx}[SCvx]{successive convexification}
\newacro{psd}[PSD]{power spectral density}
\newacro{ldt}[LDT]{Large Deviation theory}
\newacro{sde}[SDE]{stochastic differential equation}
\newacro{foh}[FOH]{first-order hold}
\newacro{crtbp}[CRTBP]{circular restricted three body problem}
\newacro{cs}[CS]{covariance steering}
\newacro{LTI}[LTI]{linear time invariant}

\newtheorem{theorem}{Theorem}

\newtheorem{lemma}{Lemma}

\title{\LARGE \bf Relative Entropy-Bounded Ambiguous Chance Constraints for Robust Planning in Nonlinear Systems}

\author{Trevor N. Wolf and Jay W. McMahon
\thanks{Trevor N. Wolf and Jay W. McMahon are with the Smead Department of Aerospace Engineering Sciences, University of Colorado Boulder, Boulder, CO 80303, USA {\tt\small trevor.wolf@colorado.edu}, {\tt\small jay.mcmahon@colorado.edu}}}

\begin{document}

\maketitle
\thispagestyle{empty}
\pagestyle{empty}

\begin{abstract}

We consider defining risk probability in stochastic control problems under distribution ambiguity. Current approaches for chance-constrained control typically assume that the true state distribution is known and Gaussian distributed. These assumptions are not amenable to many real-world engineering applications where system dynamics are nonlinear and only approximately modeled. In this work, we define a distribution ambiguity set and, with a variational expression for exponential integrals, bound the expected risk value under an unknown distribution that resides within a relative entropy distance of a nominal Gaussian reference distribution. Our bound recovers the reference risk value in the zero-divergence limit. A method is presented to determine the relative entropy distance defining the ambiguity set that is a function of the reference covariance evolution and second-order dynamical truncation errors. The resulting contributions provide a framework for handling distributional ambiguity in nonlinear covariance steering problems. A stochastic spacecraft guidance example is presented to demonstrate our contributions. 

\end{abstract}


\section{INTRODUCTION}

For designing guidance and control policies, it is necessary to account for external disturbances that act on a dynamical system and shape the evolution of the state uncertainty. Chance constraints play a role in this process by ensuring that operational criteria are met with near certainty \cite{Pu_2008}. For example, in the context of spacecraft guidance, it is often necessary to impose chance constraints on the risk probability that a spacecraft collides with either other spacecraft or the central body that it orbits \cite{Chai_2020, Oguri_2021}. Most often, these types of constraints are imposed under the pretense that the true system state distribution is known exactly, and even more restrictive, that the distribution is Gaussian. Indeed, the theory of \ac{cs} \cite{Holtz_1987}, and recent extensions that incorporate chance constraints under the \ac{cs} framework \cite{Okamoto_2018}, make these exact assumptions. The known-Gaussian assumption becomes particularly problematic for highly nonlinear dynamical systems that rapidly deteriorate Gaussianity, or if the exact dynamics of said system are unknown \cite{Ermakov_2018}. 

Recent work has investigated extensions to traditional \ac{cs} by steering Gaussian mixtures, which, in principle, more accurately quantify the true underlying uncertainty for nonlinear systems \cite{Boone_2022, Kumagai_2024}. Others forgo the problem of precise uncertainty quantification by considering ambiguity in the distribution on which the constraint is imposed -- termed ambiguous or distributionally robust chance constraints \cite{Erdogan_2006}. Underpinning all of these works is the notion that a \ac{pdf} ambiguity set can be bounded through a measurable distance between distributions \cite{Kuhn_2025}.  This distance bound is functionally used to design a control policy robust to the worst-case outcome over the set of distributions residing in that set.  Recently, the Wasserstein distance has shown particular promise for designing robust guidance policies under distribution ambiguity for \ac{LTI} systems \cite{Aolaritei_2023, Pilipovsky_2024}.  Regardless of the distance measure used to define the ambiguity set, a major challenge that has not been completely addressed is how to appropriately define a bound value. 

In this work, we use a relative entropy, or \ac{kld}, bounded ambiguity set, and propose its use in the context of the ambiguous chance constrained \ac{cs} problem. The relative entropy is particularly appealing due to its numerous connections to both information theory and dynamical systems theory. This is carried out by bounding an expected value under an unknown distribution via a form of the Donsker–Varadhan variational formula \cite{Dembo_1998}. We show that in the limit as the relative entropy converges to zero, our risk value becomes the exact risk probability computed under a nominal reference distribution. Other works have explored the use of the \ac{kld} in similar contexts. For instance, the authors in \cite{Hu_2013} consider bounding ambiguity in an objective and constraint functions through the \ac{kld} and show an equivalence to the Bernstein approximation introduced in \cite{Nemirovski_2006}. Importantly, however, this work introduces an approach to determine an appropriate value for the relative entropy upper bound defining the ambiguity set. This approach is particular applicable to nonlinear covariance steering. The value derived is functionally dependent on the covariance evolution of the problem's reference distribution, and because the reference covariance appears as a decision variable, the bound can in principle be controlled to reduce a realized distribution's departure from a Gaussian reference. This work assesses the resulting robust risk upper bound in a nonlinear spacecraft guidance design problem, providing the fundamentals needed for incorporating this metric with chance-constrained covariance steering problems. 

\subsection{Notation and Definitions}

Boldface is used to denote vectors. The subscript $t$ denotes the continuous time-dependence of a variable or function. A random variable is said to belong to the probability space $(\mathcal{X}, \mathcal{F}, \mathcal{P})$, where $\mathcal{X}$ is a complete, separable metric space, $\mathcal{F}$ is the Borel $\sigma$-algebra, and $\mathcal{P}(\mathcal{X})$ denotes the collection of probability measures on $\mathcal{X}$. If $\Pi \in \mathcal{P}(\mathcal{X})$, with an associated probability density $\pi$, then the first two central moments are defined as the mean vector $\boldsymbol{\mu} = \mathbb{E}_{\pi}[\boldsymbol{x}]$, and the covariance $P = \mathbb{E}_{\pi}[(\boldsymbol{x} - \boldsymbol{\mu})(\boldsymbol{x} - \boldsymbol{\mu})^\top]$, where $\mathbb{E}_{\pi}[\cdot]$ is the statistical expectation under the measure $\Pi$. 

\section{PRELIMINARIES}

\subsection{Model Definition and Approximation}

In this section, we discuss the model and structural approximations that are commonly introduced when applying covariance steering to nonlinear systems. We assume that the true system state dynamics are modeled by the \ac{sde}
\begin{equation}
    d \boldsymbol{x}_t = \boldsymbol{\phi}_t(\boldsymbol{x}_t, \boldsymbol{u}_t) dt + Gd\boldsymbol{w}_t, \hspace{1em} \boldsymbol{x}_0 \sim \pi_0. \label{eqn:stochastic_process}
\end{equation}
The function $\boldsymbol{\phi}_t(\cdot)$ is the true nonlinear drift that is a function of the state position and velocity, $\boldsymbol{x}_t = [\boldsymbol{r}_t^\top, \boldsymbol{v}_t^\top]^\top$, and continuous-time control vector $\boldsymbol{u}_t$. The matrix $G$ is the continuous-time process noise gain matrix, and $d\boldsymbol{w}_t \sim \mathcal{N}(\boldsymbol{0}_{n\times 1}, Q dt)$ is a standard Brownian diffusion where $Q = I_{n\times n}\delta(t - \tau)$. The initial state is distributed according to $\pi_0$, which is potentially unknown. 

We assume $\boldsymbol{\phi}_t(\cdot)$ takes the form
\begin{equation}
    \boldsymbol{\phi}_t(\boldsymbol{x}_t, \boldsymbol{u}_t) = \begin{bmatrix}
        \boldsymbol{v}_t\\
        -\nabla_{\boldsymbol{r}_t} U(\boldsymbol{r}_t; \boldsymbol{\theta}) 
    \end{bmatrix} + B\boldsymbol{u}_t.
\end{equation}
While not required, the control enters the system linearly with constant gain $B$. The scalar function $U(\boldsymbol{r}_t; \boldsymbol{\theta})$ is the potential for a conservative dynamical system, e.g., the gravitational potential of a central body, that is parameterized by the vector $\boldsymbol{\theta}$. In this scenario, $\boldsymbol{\theta}$ can define, for example, the coefficients of a spherical harmonic expansion of the gravitational potential \cite{Montenbruck_2002}. The use may vary, however, in general we assume imperfect knowledge of these parameters, and therefore, model the dynamical flow with an \textit{approximate} nonlinear drift function
\begin{equation}
    \boldsymbol{f}_t(\boldsymbol{x}_t, \boldsymbol{u}_t) = \begin{bmatrix}
        \boldsymbol{v}_t\\
        -\nabla_{\boldsymbol{r}_t} U(\boldsymbol{r}_t, \hat{\boldsymbol{\theta}})
    \end{bmatrix} + B \boldsymbol{u}_t,
\end{equation}
where $\hat{\boldsymbol{\theta}}$ is our estimate for these dynamical parameters. 


\subsection{Structural Approximations}

Motivated by the application of successive covariance steering for chance constrained trajectory optimization \cite{Ridderhof_2022, Kumagai_2025}, we seek a Gaussian time-varying reference distribution $\hat{\pi}_t$ by imposing the structural approximations that 1) the initial \ac{pdf} is Gaussian distributed, and 2) its drift is approximated by a linearization of the model $\boldsymbol{f}_t$. To that end, the reference state \ac{sde} is 
\begin{equation}
    d\boldsymbol{x}_t = \hat{\boldsymbol{f}}_t(\boldsymbol{x}_t, \boldsymbol{u}_t) + G d\boldsymbol{w}_t, \hspace{1em} \boldsymbol{x}_0 \sim \mathcal{N}(\boldsymbol{\mu}_0, P_0).
\end{equation}
The approximate drift is the affine function 
\begin{equation}
    \hat{\boldsymbol{f}}_t(\boldsymbol{x}_t, \boldsymbol{u}_t) = A_t\boldsymbol{x}_t + B \boldsymbol{u}_t + \boldsymbol{r}_t,
\end{equation}
where
\begin{equation}
    A_t = \nabla_{\boldsymbol{x}_t}\boldsymbol{f}_t(\boldsymbol{\mu}_t, \boldsymbol{u}_t),
\end{equation}
and 
\begin{equation}
    \boldsymbol{r}_t = \boldsymbol{f}_t(\boldsymbol{\mu}_t, \boldsymbol{u}_t) - A_t\boldsymbol{\mu}_t - B\boldsymbol{u}_t.
\end{equation}

\section{A CHANCE CONSTRAINT VARIATIONAL UPPER BOUND UNDER DISTRIBUTION AMBIGUITY}

We introduce an approach for bounding chance constraints under distribution ambiguity — otherwise known as distributionally robust chance constraints — by using a duality relation between exponential integrals and relative entropy. As an aside, in the uncertainty quantification literature, distribution ambiguity is sometimes referred to as epistemic uncertainty, and our contributions are motivated by works developed under this nomenclature \cite{Chowdhary_2013, Chowdhary_2015, Hofman_2024}. 

Ambiguity in the time-evolving state \ac{pdf} can arise from two sources: 1) lack of knowledge in the model parameters governing the dynamical system, and 2) structural approximations imposed by linear-Gaussian assumptions.  A chance constraint under the \textit{true}, but unknown, probability distribution $\pi_t$ is formally 
\begin{equation}
    \mathbb{P}_{\boldsymbol{x}_t \sim \pi_t}[\boldsymbol{x}_t \in \mathcal{S}_t] \geq 1 - \epsilon \Longleftrightarrow \mathbb{P}_{\boldsymbol{x}_t \sim \pi_t}[\boldsymbol{x}_t \notin \mathcal{S}_t] < \epsilon,\label{eqn:chance_constraint}
\end{equation}
where $\mathcal{S}_t$ is a time-evolving envelope of acceptable states, and $\epsilon$ is a small probability value (e.g. $\epsilon = 0.01$). The latter form in (\ref{eqn:chance_constraint}), also known as the risk or violation probability, is equivalently expressed as 
\begin{equation}
    \mathbb{P}_{\boldsymbol{x}_t \sim \pi_t}[\boldsymbol{x}_t \notin \mathcal{S}_t] = \mathbb{E}_{\pi_t}[1 - \mathds{1}_{\{\boldsymbol{x}_t \in \mathcal{S}_t \}}],
\end{equation}
where $\mathds{1}_{\{ \cdot \}}$ is the indicator function. 
\begin{theorem}
    Consider the probability measures $P$ and $Q$ that are absolutely continuous, and with associated densities $p$ and $q$. Then, for any bounded function $g(\boldsymbol{x}): \mathcal{X} \rightarrow \mathbb{R}$, 
    \begin{equation}
        \frac{1}{\lambda} \ln \left(\mathbb{E}_{q}\left[e^{\lambda g(\boldsymbol{x})} \right]\right) = \sup_{P\in \mathcal{P}(\mathcal{X})}\left[-\frac{1}{\lambda} R(p\|q) + \mathbb{E}_p[g(\boldsymbol{x})] \right],
    \end{equation}
    if relative entropy $R(p\|q) < \infty$ and $\lambda$ is any positive scalar.\label{theorem:DV_formula}
\end{theorem}

For a proof of Theorem \ref{theorem:DV_formula}, see \cite{Dupuis_1997}, Proposition 4.5.1. Setting $g(\cdot) = 1 - \mathds{1}_{\{\boldsymbol{x}_t \in \mathcal{S}_t\}} $, it follows that
\begin{flalign}
    \mathbb{P}_{\boldsymbol{x}_t \sim \pi_t}[\boldsymbol{x}_t \notin \mathcal{S}_t] &\leq \frac{1}{\lambda}\Big\{\ln\left(\mathbb{E}_{\hat{\pi}_t}\left[e^{\lambda(1 - \mathds{1}_{\{\boldsymbol{x}_t \in \mathcal{S}_t\}})}\right] \right)\nonumber\\
    & \hspace{6em}+ R(\pi_t\|\hat{\pi}_t)\Big\}\nonumber\\
    & \hspace{-5em} = 1 + \frac{1}{\lambda}\Big\{\ln\left(1 - (1 - e^{-\lambda})\mathbb{P}_{\boldsymbol{x}_t \sim \hat{\pi}_t}[\boldsymbol{x}_t \in \mathcal{S}_t] \right) \nonumber \\
    & \hspace{6em}+ R(\pi_t\|\hat{\pi}_t) \Big\}
    \label{eqn:chance_upper_bound}
\end{flalign}
The expression in (\ref{eqn:chance_upper_bound}) provides an upper bound for the true probability of violation that depends on a time-varying reference distribution $\hat{\pi}_t$, and the relative entropy, or \ac{kld}, between the true distribution and the reference:
\begin{equation}
    R(\pi_t\|\hat{\pi}_t) = \int_{\mathcal{X}} \pi_t \ln\left(\frac{\pi_t}{\hat{\pi}_t} \right)d\boldsymbol{x}_t.
    \label{eqn:rel_entropy}
\end{equation}
In this work, we let the admissible state set to be of the form 
\begin{equation}
    \mathcal{S}_t = \{ \boldsymbol{x}_t \in \mathbb{R}^n | \boldsymbol{a}_t^\top \boldsymbol{x}_t \leq b_t \},
\end{equation}
and our reference distribution to be Gaussian, then, 
\begin{flalign}
    \mathbb{P}_{\boldsymbol{x}_t \sim \pi_t}[\boldsymbol{x}_t \notin \mathcal{S}_t] &\leq 1 + \frac{1}{\lambda}\Bigg\{\ln\Bigg(1 - (1 - e^{-\lambda})\nonumber\\
    &\operatorname{cdf}\Bigg(\frac{b_t - \boldsymbol{a}_t^\top \boldsymbol{\mu}_t}{\sqrt{\boldsymbol{a}_t^\top P_t \boldsymbol{a}_t}} \Bigg) \Bigg)  + R(\pi_t\|\hat{\pi}_t) \Bigg\},
    \label{eqn:risk_upper_bound}
\end{flalign}
where the function $\operatorname{cdf}(\cdot)$ is the standard normal \ac{cdf}. In the above, the scalar $\lambda$ serves as a tuning variable controlling the tightness of the upper bound. There exists a unique minimizer $\lambda^*$ for (\ref{eqn:risk_upper_bound}) that depends on the nominal probability function and the relative entropy. In the limit as minimizer $\lambda^* \rightarrow 0$, the corresponding relative entropy tends towards $R(\pi_t\|\hat{\pi}_t) = 0$, i.e., $\pi_t = \hat{\pi}_t$, and  
\begin{flalign}
    \lim_{\lambda^* \rightarrow 0} 1 + & \frac{1}{\lambda^*}\Bigg\{\ln\Bigg(1 - (1 - e^{-\lambda^*}) \operatorname{cdf}\Bigg(\frac{b_t - \boldsymbol{a}_t^\top \boldsymbol{\mu}_t}{\sqrt{\boldsymbol{a}_t^\top P_t \boldsymbol{a}_t}} \Bigg) \Bigg)\nonumber\\
    &+ R(\pi_t\|\hat{\pi}_t) \Bigg\} = 1 - \operatorname{cdf}\Bigg(\frac{b_t - \boldsymbol{a}_t^\top \boldsymbol{\mu}_t}{\sqrt{\boldsymbol{a}_t^\top P_t \boldsymbol{a}_t}} \Bigg), 
\end{flalign}
which is the exact risk probability under a Gaussian reference distribution. The above follows from the fact that for small $\lambda$, the expression $1 - e^\lambda = \lambda + \mathcal{O}(\lambda^2)$, and consequently $\ln(1 - \lambda\operatorname{cdf(\cdot)}) = -\lambda\operatorname{cdf(\cdot)} + \mathcal{O}(\lambda^2)$  In Figure \ref{fig:risk_vs_lambda} we show the risk upper bound as a function of $\lambda$, evaluated with different values of distribution ambiguities $R(\pi_t\|\hat{\pi}_t)$.

\begin{figure}[t!]
    \centering
    \includegraphics[width=0.4\textwidth]{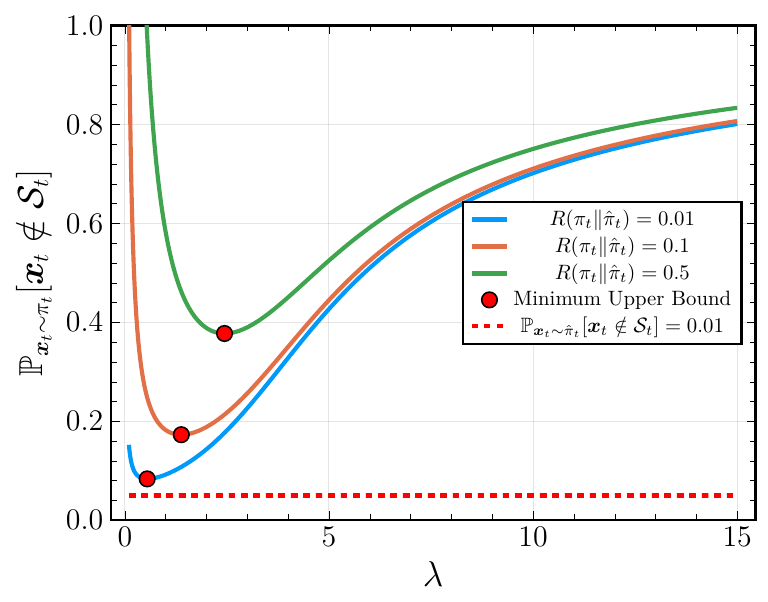}
    \caption{Evaluation of the risk upper bound under distribution ambiguity as a function of $\lambda$.}\label{fig:risk_vs_lambda}
\end{figure}

\section{CONTROLLING THE RELATIVE ENTROPY}

The previous section provides an upper bound for the true constraint violation probability that depends on a reference \ac{pdf} and a relative entropy upper bound between the true \ac{pdf} and the assumed reference. Here, we provide that time-varying relative entropy upper bound that, in the same vein as before, does not require explicit knowledge of the true \ac{pdf}. We achieve this by integrating an upper bound for the relative entropy growth rate, and we derive that rate function below. 
\begin{lemma}
    If the true and reference densities $\pi_t$ and $\hat{\pi}_t$ satisfy the Fokker-Planck equation with the same diffusion, then 
    \begin{equation}
        \frac{d}{dt}R(\pi_t\|\hat{\pi}_t) \leq \frac{1}{2} \mathbb{E}_{\pi_t} \left[ \| D^{-1/2} (\boldsymbol{\phi}_t - \hat{\boldsymbol{f}}_t)\|_2^2 \right], 
    \end{equation}
    where the matrix $D = GQG^\top$.
    \label{lem:kld_rate}
\end{lemma}

\textit{Proof of Lemma \ref{lem:kld_rate}:} This proof is similar and follows closely with that provided in  \cite{Mou_2022}. Note that in the above, and from here on, we drop the dependence of $\boldsymbol{x}_t$ and $\boldsymbol{u}_t$ in the drift functions for conciseness, but it should be understood that dependence is implied. 

First, observe that 
\begin{flalign}
    \frac{\partial}{\partial t}\bigg(\pi_t \ln \frac{\pi_t}{\hat{\pi}_t} \bigg) &= \pi_t \bigg(\frac{\partial \ln \pi_t}{\partial t}(\ln \pi_t + 1 - \ln \hat{\pi}_t)\nonumber\\ 
    &\hspace{6em}- \frac{\partial \ln \hat{\pi}_t}{\partial t} \bigg).\label{eqn:proof_2_1}
\end{flalign}
Under mild assumptions, i.e., dominated convergence of $\pi_t\ln(\pi_t/\hat{\pi}_t)$, we can differentiate (\ref{eqn:rel_entropy}) by switching the order of integration and differentiation. Then, substituting (\ref{eqn:proof_2_1}), 
\begin{flalign}
    \frac{d}{dt} R(\pi_t\| \hat{\pi}_t) &= \int_{\mathcal{X}} \frac{\partial \pi_t}{\partial t}(\ln \pi_t + 1 - \ln \hat{\pi}_t) d\boldsymbol{x}_t \\
    & \hspace{4em}- \int_{\mathcal{X}} \frac{\partial \hat{\pi}_t}{\partial t} \frac{\pi_t}{\hat{\pi}_t} d\boldsymbol{x}_t. \label{eqn:proof_2_2}
\end{flalign}
The time evolution of $\pi_t$ is given by the Fokker-Planck equation,
\begin{equation}
    \frac{\partial \pi_t}{\partial t} = -\nabla_{\boldsymbol{x}_t} \cdot( \pi_t \boldsymbol{\phi}_t) + \frac{1}{2} \nabla_{\boldsymbol{x}_t} \cdot(D\nabla_{\boldsymbol{x}_t} \pi_t),
    \label{eqn:proof_2_3}
\end{equation}
and similarly for $\hat{\pi}_t$ by exchanging the $\boldsymbol{\phi}_t$ with $\hat{\boldsymbol{f}}_t$. By substituting in (\ref{eqn:proof_2_3}) and applying the divergence theorem, the first term in (\ref{eqn:proof_2_2}) is 
\begin{equation}
    - \int_{\mathcal{X}} \pi_t\bigg(-\boldsymbol{\phi}_t +  \frac{1}{2} D \nabla_{\boldsymbol{x}_t} \ln \pi_t\bigg)\cdot(\nabla_{\boldsymbol{x}_t}\ln \pi_t - \nabla_{\boldsymbol{x}_t} \ln \hat{\pi}_t) d\boldsymbol{x}_t,
\end{equation}
and the second is 
\begin{equation}
    \int_{\mathcal{X}} \pi_t\bigg(-\hat{\boldsymbol{f}} + \frac{1}{2}D\nabla_{\boldsymbol{x}_t}\hat{\pi}_t\bigg)\cdot(\nabla_{\boldsymbol{x}_t} \ln \pi_t - \nabla_{\boldsymbol{x}_t}\hat{\pi}_t)d\boldsymbol{x}_t.
\end{equation}
Combining these two terms and reducing,
\begin{flalign}
    \frac{d}{dt} R(\pi_t\| \hat{\pi}_t) &= \int_{\mathcal{X}} \pi_t(\boldsymbol{\phi}_t - \hat{\boldsymbol{f}}_t) \cdot (\nabla_{\boldsymbol{x}_t} \ln \pi_t - \nabla_{\boldsymbol{x}_t} \ln \hat{\pi}_t) d \boldsymbol{x}_t\nonumber\\
    & \hspace{2em} - \frac{1}{2} \int_{\mathcal{X}} \|D^{1/2}(\nabla_{\boldsymbol{x}_t} \ln \pi_t - \nabla_{\boldsymbol{x}_t} \hat{\pi}_t)\|_2^2 d\boldsymbol{x}_t.
\end{flalign}
By Young's inequality, i.e., $\boldsymbol{a}\cdot \boldsymbol{b} \leq \frac{1}{2}(\|\boldsymbol{a}\|_2^2 + \|\boldsymbol{b}\|_2^2)$, it is straightforward to show that
\begin{flalign}
    \frac{d}{dt} R(\pi_t\|\hat{\pi}_t) &\leq \frac{1}{2}\int_{\mathcal{X}} \pi_t\|D^{-1/2}(\boldsymbol{\phi}_t - \hat{\boldsymbol{f}}_t)\|_2^2 d\boldsymbol{x}_t\nonumber\\
    &=\frac{1}{2}\mathbb{E}_{\pi_t}\left[\|D^{-1/2}(\boldsymbol{\phi}_t - \hat{\boldsymbol{f}}_t)\|_2^2 \right].\label{eqn:proof_2_final}
\end{flalign}
\hfill $\square$

Notice that the expression in (\ref{eqn:proof_2_final}) still requires the expectation under the true law, $\pi_t$, which is unknown. We can remedy this by bounding the discrepancy in the laws' expected values by again invoking Theorem \ref{theorem:DV_formula}. For conciseness, let $g_t = \|D^{-1/2}(\boldsymbol{\phi}_t - \hat{\boldsymbol{f}}_t)\|_2^2$. The discrepancy has the upper bound
\begin{flalign}
    \mathbb{E}_{\pi_t}[g_t] - \mathbb{E}_{\hat{\pi}_t}[g_t] &\leq \frac{1}{\lambda}\big\{\ln\left(\mathbb{E}_{\hat{\pi}_t}\left[e^{\lambda g_t} \right] \right) + R(\pi_t\|\hat{\pi}_t) \nonumber\\ \nonumber
    &\hspace{6em}-\lambda \mathbb{E}_{\hat{\pi}}[g_t] \big\}\\
    & \hspace{-3em} = \frac{1}{\lambda}\left\{\ln\left(\mathbb{E}_{\hat{\pi}_t}\left[e^{\lambda(g_t - \mathbb{E}_{\hat{\pi}}[g_t])}\right] \right) + R(\pi_t\|\hat{\pi}_t) \right\} 
\end{flalign}
for some $\lambda>0$. Using Hoeffding's Lemma \cite{Hoeffding_1963}, i.e,  for a distribution $q$, and the random variable $a \leq X \leq b$, 
\begin{equation}
    \mathbb{E}_{q}\left[e^{\lambda(X - \mathbb{E}_q[X]}\right] \leq e^{\frac{\lambda^2(b - a)^2}{8}},
\end{equation}
it follows that 
\begin{equation}
    \mathbb{E}_{\pi_t}[g_t] - \mathbb{E}_{\hat{\pi}_t}[g_t] \leq \frac{1}{8}\lambda(\overline{g_t} - \underline{g_t})^2 + \frac{1}{\lambda}R(\pi_t\| \hat{\pi}_t)\label{eqn:descripancy_difference},
\end{equation}
where $\overline{g_t}$ and $\underline{g_t}$ are defined upper and lower bound values of $g_t$, respectively. The minimizer of the right-hand side of (\ref{eqn:descripancy_difference}) is given by 
\begin{equation}
    \lambda^* = \sqrt{\frac{8R(\pi_t\|\hat{\pi}_t)}{(\overline{g_t} - \underline{g_t})^2}}.
\end{equation}
Putting the pieces together, the tightest upper bound for the relative entropy growth rate is 
\begin{flalign}
    \frac{d}{dt} R(\pi_t\|\hat{\pi}_t) &\leq \frac{1}{2}\mathbb{E}_{\pi_t}[g_t]\nonumber\\
    &= \frac{1}{2}\left(\mathbb{E}_{\hat{\pi}_t}[g_t] + (\mathbb{E}_{\pi_t}[g_t] - \mathbb{E}_{\hat{\pi}_t}[g_t]) \right)\nonumber\\
    &\leq \frac{1}{2}\left(\mathbb{E}_{\hat{\pi}_t}[g_t] + \frac{\overline{g_t} - \underline{g_t}}{\sqrt{2}} \sqrt{R(\pi_t\|\hat{\pi}_t)} \right).
    \label{eqn:tightest_upper_bound}
\end{flalign}
Note that in general, there does not exist a finite upper bound $\overline{g}_t$. In practical settings, one should use a suitable heuristic to determine $\overline{g}_t$ that is informed by the problem at hand. The following will introduce one such heuristic that we choose to use in this study, but Eq. \ref{eqn:tightest_upper_bound} is not limited to this choice.  

\subsection{Implications of a Gaussian Reference}

In the present study, we assume that the model $\boldsymbol{f}_t(\cdot) = \boldsymbol{\phi}_t(\cdot)$, meaning we have perfect knowledge of our dynamic model's parameters. Planned work intends to consider cases where we may know the parameters $\boldsymbol{\theta}$ probabilistically or within some bounds. Here, however, ambiguity in the distribution $\pi_t$ is solely due to the aforementioned structural approximations -- initial Gaussianity of the reference distribution and linearization of the dynamics in $\hat{\boldsymbol{f}}_t$. Furthermore, we assume that linearization errors can be approximately bounded through a quadratic Taylor remainder, centered at the reference distribution's mean: 
\begin{flalign}
    \boldsymbol{f}_t - \hat{\boldsymbol{f}}_t &= \frac{1}{2} \delta\boldsymbol{x_t}^\top \nabla_{\boldsymbol{x}_t}^2 \boldsymbol{f}_t(\boldsymbol{\mu}_t) \delta\boldsymbol{x}_t + \mathcal{O}(\|\delta\boldsymbol{x}_t\|^3)\nonumber\\
    & \simeq \frac{1}{2}\left[\delta\boldsymbol{x}_t^\top H_t^1 \delta\boldsymbol{x}_t, \cdots, \delta\boldsymbol{x}_t^\top H_t^n \delta\boldsymbol{x}_t\right]^\top.
\end{flalign}
In the above, $ \delta \boldsymbol{x}_t = \boldsymbol{x}_t - \boldsymbol{\mu}_t$ and the symmetric Hessian matrix $H_t^i = \nabla^2_{\boldsymbol{x}_t} f_{t}^i(\boldsymbol{\mu}_t)$, where $f_{t}^i$ is the $i^{\text{th}}$ component of the nonlinear drift model vector $\boldsymbol{f}_t$. Furthermore, if we assume the process noise covariance is spherical, i.e., $D = \sigma^2I_{n\times n}$, then
\begin{flalign}
    \mathbb{E}_{\hat{\pi}_t}[g_t] &= \mathbb{E}_{\hat{\pi}_t}\left[\|D^{-1/2}(\boldsymbol{f}_t - \hat{\boldsymbol{f}}_t)\|_2^2 \right]\nonumber\\
    &= \frac{1}{4\sigma^2} \mathbb{E}_{\hat{\pi}_t}\left[\sum_{i = 1}^n \left(\delta\boldsymbol{x}_t^\top H_t^i \delta\boldsymbol{x}_t\right)^2 \right] \nonumber \\
    & = \frac{1}{4\sigma^2} \sum_{i = 1}^n \mathbb{E}_{\hat{\pi}_t}[(\delta \boldsymbol{x}_t^\top H_t^i \delta \boldsymbol{x}_t)^2].
    \label{eqn:reduce_expectation}
\end{flalign}

\begin{lemma}
    Let $S \in \mathbb{R}^{n\times n}$ and $T \in \mathbb{R}^{n\times n}$ be a symmetric matrices and $\boldsymbol{\xi} \sim \mathcal{N}(\boldsymbol{0}_{n\times 1}, P)$, where $P\in \mathbb{R}^{n\times n}$ is a positive definite covariance matrix, then
    \begin{equation}
        \mathbb{E}[(\boldsymbol{\xi}^\top S \boldsymbol{\xi})(\boldsymbol{\xi}^\top B \boldsymbol{\xi})] = \text{tr}(SP)\,\text{tr}(TP) + 2\,\text{tr}(SPTP),
    \end{equation}
    and
    \begin{equation}
        \mathbb{E}[(\boldsymbol{\xi}^\top S \boldsymbol{\xi})^2] = \text{tr}(SP)^2 + 2\,\text{tr}\left((SP)^2\right).
    \end{equation}
    \label{lem:trace_expression}
\end{lemma}
For a proof of Lemma \ref{lem:trace_expression} see Theorem 4.2 in \cite{Magnus_1979}. Using this lemma with (\ref{eqn:reduce_expectation}), we find that
\begin{equation}
    \mathbb{E}_{\hat{\pi}_t}[g_t] = \frac{1}{4\sigma^2}\sum_{i = 1}^n \left[ \text{tr}(H^i_t P_t)^2 + 2\,\text{tr}\left((H_t^iP_t)^2 \right) \right].\label{eqn:guassian_approx_KLD_rate}
\end{equation}

The lower bound $\underline{g_t}$ is trivially zero. Determining the upper bound $\overline{g_t}$, however, is problem dependent. In general, there does not exist a finite upper bound on the approximation truncation error terms $\delta \boldsymbol{x}_t^\top H^i_t \delta \boldsymbol{x}_t$, and the distribution of $(\delta \boldsymbol{x}_t^\top H^i_t  \delta \boldsymbol{x}_t)^2$ is not sub-Gaussian or sub-exponential, so care must be taken in choosing $\overline{g_t}$ to properly ensure that it bounds $g_t$ to a reasonable degree of certainty. In the present study, we make a simple assumption and set $\overline{g_t} = (1 + \gamma) \mathbb{E}_{\hat{\pi}_t}[g_t] $, where $\gamma > 1$. Future work intends to find a formal approach to determine $\overline{g_t}$.

Putting the pieces together, we find the upper bound for the relative entropy growth rate 
\begin{flalign}
    \frac{d}{dt} R(\pi_t\| \hat{\pi}_t) &\leq \frac{1}{2} \mathbb{E}_{\hat{\pi}_t}[g_t]\left(1 + \frac{1 + \gamma}{\sqrt{2}} 
    \sqrt{R(\pi_t\|\hat{\pi}_t)} \right)\nonumber\\
    &\hspace{-2em} = \frac{1}{8\sigma^2}\sum_{i = 1}^n\left[\operatorname{tr}(H_t^i P_t)^2 + 2\operatorname{tr}\left((H_t^iP_t)^2\right) \right]\nonumber\\
    & \cdot \left(1 + \frac{1 + \gamma}{\sqrt{2}} 
    \sqrt{R(\pi_t\|\hat{\pi}_t)} \right)
    \label{eqn:rate_final_form}
\end{flalign}
The right-hand side in (\ref{eqn:rate_final_form}) is independent of knowing the true distribution $\pi_t$, and requires only the mean and covariance of $\hat{\pi}_t$ and a running evaluation of $R(\pi_t\|\hat{\pi}_t)$, or, because of (\ref{eqn:rate_final_form}) being strictly positive, a running upper bound provided by integrating the expression itself.

\section{NUMERICAL IMPLEMENTATION}
In the present study, we consider the linearized covariance steering problem as a use case to evaluate our contributions. The problem is widely found in the stochastic optimization literature (e.g., \cite{Okamoto_2018, Ridderhof_2022}), so here we only briefly summarize. We seek to find a control feedback parameterized by
\begin{equation}
    \delta\boldsymbol{u}_t = K_t\delta\boldsymbol{x}_t.
    \label{eqn:linear_feedback}
\end{equation}
The continuous-time linearized state deviation flow is then
\begin{flalign}
    \delta \dot{\boldsymbol{x}}_t &= A_t\delta\boldsymbol{x}_t + B\delta \boldsymbol{u_t} + G\boldsymbol{w}_t\\
    &=(A_t + BK_t)\delta\boldsymbol{x}_t + G\boldsymbol{w}_t.
\end{flalign}
For implementation in a convex static optimization problem, we use a \ac{foh} interpolation scheme so the discrete time state deviation transition and its associated covariance between time steps $t_k$ and $t_{k + 1}$ are 
\begin{equation}
    \delta\boldsymbol{x}_{k + 1} = (I - B_k^+ K_{k + 1})^{-1}[(A_k + B_k^- K_k)\delta\boldsymbol{x}_k + \boldsymbol{w}_k]
\end{equation}
and
\begin{flalign}
&P_{k+1} - P_{k+1} K_{k+1}^\top B_k^{+\top} - B_k^+ K_{k+1} P_{k+1}\nonumber\\
& \hspace{3em}+ B_k^+ K_{k+1} P_{k+1} K_{k+1}^\top B_k^{+\top} \nonumber \\
&\hspace{1em}= A_k P_k A_k^\top + A_k P_k K_k^\top B_k^{-\top} \nonumber \\
&\hspace{3em}+ B_k^- K_k P_k A_k + B_k^- K_k P_k K_k^\top B_k^{-\top} + Q_k. 
\label{eqn:covariance_evolution}
\end{flalign}
The matrices in the above are 
\begin{flalign}
     A_k &= \Phi(t_{k + 1}, t_k)\mbox{,}\\
     B_k^{-} &= A_k\int_{t_k}^{t_{k + 1}} \frac{t_{k + 1} - \tau}{t_{k + 1} - t_k} \Phi^{-1}(\tau, t_k)
     B d\tau \mbox{,}\\
     B_k^{+} &= A_k\int_{t_k}^{t_{k + 1}}\frac{\tau - t_k}{t_{k + 1} - t_k} \Phi^{-1}(\tau, t_k) B d\tau,\\
     Q_{k} &= \int_{t_k}^{t_{k + 1}} \Phi(t_{k + 1}, \tau) GG^\top \Phi^{\top}(t_{k + 1}, \tau) d\tau. 
\end{flalign}
The matrix $\Phi(t_{k + 1}, t_k)$ is the state transition matrix between time steps $k$ and $k + 1$. 

Covariance dynamics constraints by enforcing (\ref{eqn:covariance_evolution}) are nonconvex in the decision variables $K_k$ and $P_k$. Following \cite{Rataczak_2025}, we use a convex relaxation of these constraints through the transformations 
\begin{flalign}
    U_k &= K_kP_k,\\
    Y_k &= U_k P_k^{-1}U_k^\top.
\end{flalign}
An equivalent convex state covariance transition constraint is then 
\begin{flalign}
P_{k+1} &= A_k P_k A_k^\top + A_k U_k^\top B_k^{-\top}\nonumber\\
&+ B_k^- U_k A_k^\top + B_k^- Y_k B_k^{-\top}
+ U_{k+1}^\top B_k^{+\top}\nonumber\\
&+ B_k^+ U_{k+1} - B_k^+ Y_{k+1} B_k^{+\top} + Q_k,\label{eqn:convex_covariance_evolution}
\end{flalign}
with the added matrix inequality constraint,
\begin{equation}
    \begin{bmatrix}
    P_k & U_k^\top\\
    U_k & Y_k
    \end{bmatrix} \succeq 0 \label{eqn:psd_constraint}.
\end{equation}
Determining the optimal feedback control gain is found by minimizing
\begin{flalign}
    J =  \sum_{k = 1}^{N - 1}\frac{\Delta t_k}{2}\left(\operatorname{tr}(Y_k) + \operatorname{tr}(Y_{k + 1}) \right) 
\end{flalign}
subject to (\ref{eqn:convex_covariance_evolution} - \ref{eqn:psd_constraint}) and initial and terminal constraints on the covariance, i.e., $P_0 = P_0^*$, $P_N = P_N^*$. Here $\Delta t_k = t_{k + 1} - t_k$ and $N$ is the number of discretization nodes. The gains are recovered in post-processing by $K_k = U_kP_k^{-1}$.

\subsection{Discrete Time Approximation of the Relative Entropy Upper Bound}

We use a first-order Euler method to numerically integrate the relative entropy bound. This is given by
\begin{equation}
    \overline{R}_{k + 1} = \overline{R}_k + \Delta t_k \overline{\dot{R}}_k,
\end{equation}
where $\overline{R}_k$ is the integrated upper bound at time step $t_k$ and $\overline{\dot{R}}_k$ is the discrete time rate. Using the rate function found by (\ref{eqn:rate_final_form}),
\begin{flalign}
    \overline{\dot{R}_k} &= \frac{1}{8\sigma^{2}}\sum_{i = 1}^n\left[\operatorname{tr}(H_k^i P_k)^2 + 2\operatorname{tr}\left((H_k^iP_k)^2\right) \right]\nonumber\\
    & \cdot \left(1 + \frac{1 + \gamma}{\sqrt{2}} 
    \sqrt{\overline{R}_k} \right).
    \label{eqn:rate_final_form_discrete}
\end{flalign}
In the above, $ H_k^i = \nabla^2\boldsymbol{f}_t(\boldsymbol{\mu}_k)$, where $\boldsymbol{\mu}_k$ is the nominal mean of our Gaussian reference distribution at time step $t_k$.

\section{RESULTS}
To assess our methods, we consider a nonlinear stochastic guidance problem for a spacecraft operating in the Earth-Moon dynamical region. Our main objective presently is to quantify the differences in the risk allocated by assuming a nominal Gaussian distribution and the distributionally robust risk upper bound proposed. 

\subsection{Equations of Motion}

The drift dynamics $\boldsymbol{f}(\cdot)$ are modeled with the \ac{crtbp} that has the equations of motion
\begin{flalign}
    \ddot{x} &= 2\dot{y} + x - (1 - \mu)\frac{x + \mu}{r_1^3} - \mu \frac{x + \mu - 1}{r_2^3},\nonumber\\
    \ddot{y} &= -2\dot{x} + y - (1 - \mu)\frac{y}{r_1^3} - \mu \frac{y}{r_2^3},\nonumber\\
    \ddot{z} &= -(1 - \mu)\frac{z}{r_1^3} - \mu \frac{z}{r_2^3}.
    \label{eqn:CRTBP_eom}%
\end{flalign}
The quantities $r_1$ and $r_2$ are defined as
\begin{flalign}
    r_1^2 &= (x + \mu)^2 + y^2 + z^2, \nonumber\\
    r_2^2 &= (x + \mu - 1)^2 + y^2 + z^2.
\end{flalign}
\label{eqn:CRTBP_eom}%
The parameter $\mu = \frac{m_2}{m_1 + m_2}$, where $m_1$ is the mass of the Earth, and $m_2$ is the mass of the Moon. Additive diffusion is added to the drift by sampling $d\boldsymbol{w}_t \sim \mathcal{N}(0, \sigma^2I_{3\times 3} dt)$, and the choice of $\sigma$ is described further below. 

\subsection{Scenario Description}

\begin{table}[t!]
\caption{Scenario Parameters}
\label{tab:scenario_parameters}
\centering
\begin{tabular}{ccc}
\hline
Parameter & Symbol \& Units & Value \\ \hline
Gravitational Parameter & $\mu$ & 0.012151\\
Number of Nodes & $N$ & 150\\ 
Number of Monte Carlo Trials & $N^{\text{MC}}$ & 4000\\
Diffusion Magnitude & $\sigma^{(1)}$ [$\text{km}/\text{s}^{3/2}$] & $5\times10^{-6}$\\
-- -- & $\sigma^{(2)}$ [$\text{km}/\text{s}^{3/2}$] & $1\times 10^{-6}$ \\
Initial Position Uncertainty  & $\sigma^*_r$ [km] & 600\\
Initial Velocity Uncertainty & $\sigma^*_v$[km/s] & $6\times10^{-3}$\\
Targeted Position Uncertainty & $\sigma^*_r$ [km] & 600\\
Targeted Velocity Uncertainty & $\sigma^*_v$[km/s] & $6\times10^{-3}$\\
Scenario Period & $T$ [TU] & 3.7247\\
Keep-Out Radius & $r_{\text{min}}$ [km] & 100.0
\\ \hline
\end{tabular}
\end{table}

%
\begin{figure}[htb]
    \centering
    \includegraphics[width=0.4\textwidth]{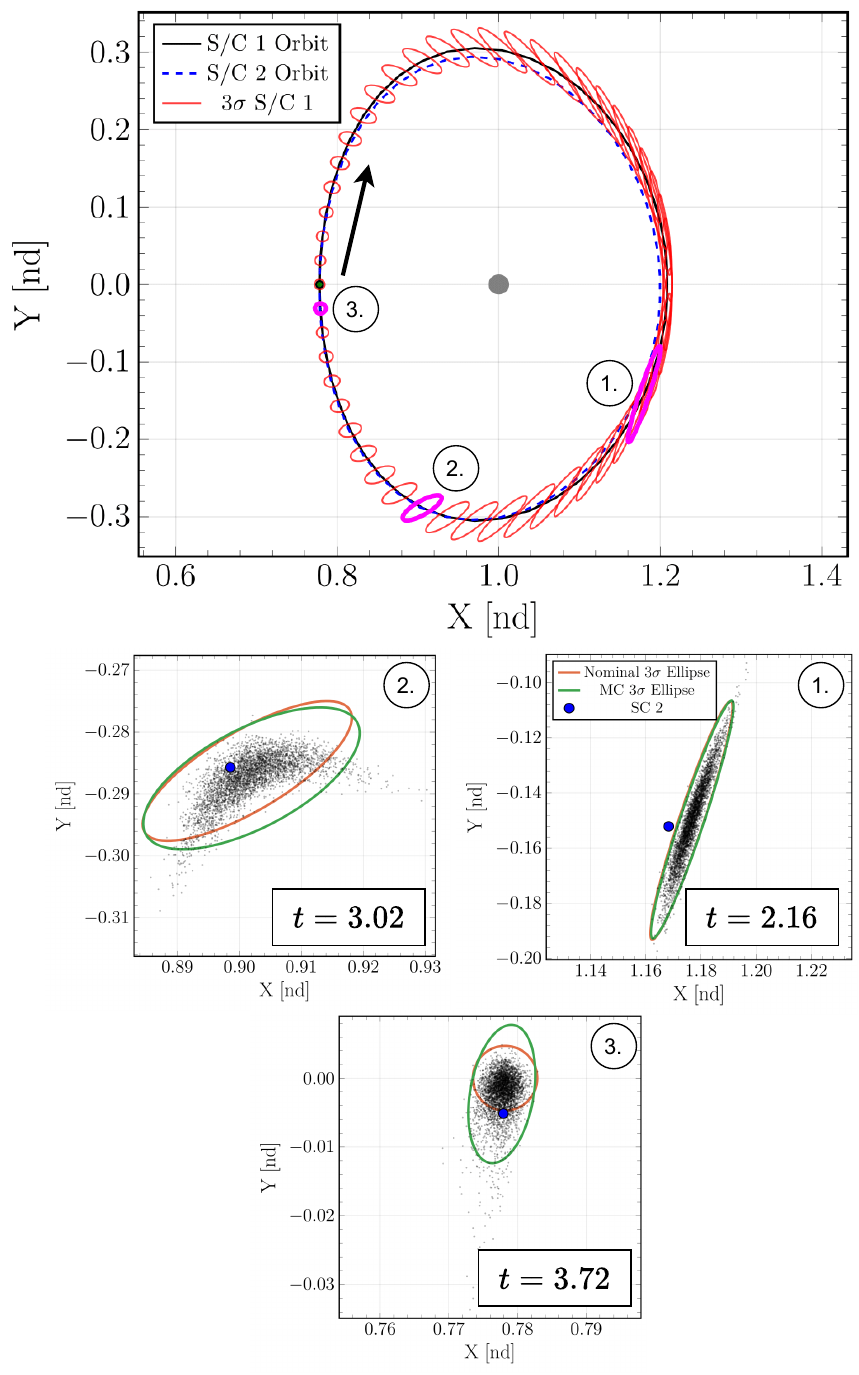}
    \caption{The top panel shows the steered nominal covariance evolution for one period of the reference orbit. The black line is the reference mean, and the blue is the orbit of a chaser spacecraft that poses a risk of collision. Three different time instances are emphasized in the bottom panels. A Monte Carlo evaluates the validity of the steered Gaussian reference.}\label{fig:orbits_and_MC}
\end{figure}
In this study's scenario, the spacecraft's initial distribution is prescribed by $\hat{\pi}_0 = \mathcal{N}(\boldsymbol{\mu}_0, P_0)$ where $\boldsymbol{\mu}_0 = [7.7819\times10^{-1},\; 0,\; 0,\; 0,\; 5.5593\times10^{-1},\; 0]^\top$ and $P_0$ is diagonal covariance with the uncertainty values provided in Table \ref{tab:scenario_parameters}. The guidance policy steers the terminal reference distribution to $\hat{\pi}_0$ after one period of the reference mean's orbit. A nominal risk is determined by the chance that the spacecraft's position is within a $r_{\text{min}}$ of a chaser spacecraft, i.e., $\mathbb{P}_{\boldsymbol{x}_t\sim\hat{\pi}_t}[\|\boldsymbol{\mu}_{r, t} - \boldsymbol{r}_{c, t}\|_2 \leq r_{\text{min}}]$. We find that diffusion plays a strong influence on the magnitude of the relative entropy, and thus the risk upper bound, so we evaluate our results using two different diffusion magnitude values (i.e, $\sigma$). Relevant scenario parameters are provided in Table \ref{tab:scenario_parameters}.  
%
\begin{figure*}[htb!]
    \centering
    \begin{subfigure}{0.35\textwidth}
        \centering
        \includegraphics[width=\linewidth]{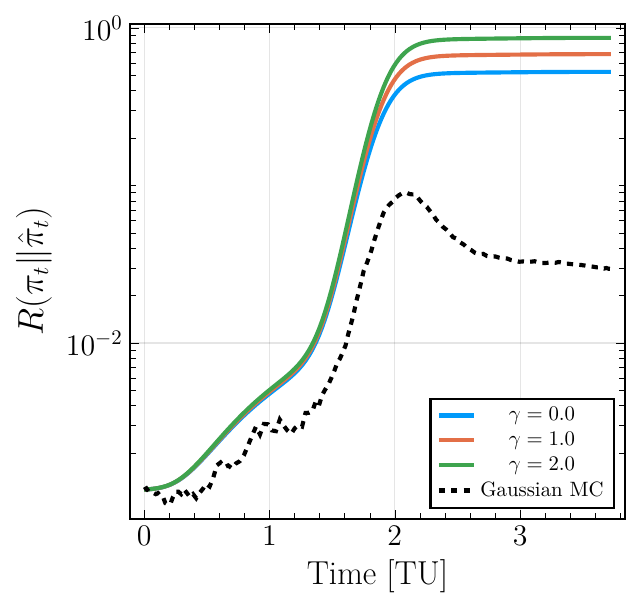} %
        \caption{}
    \end{subfigure}
    \hspace{1ex}
    \begin{subfigure}{0.35\textwidth}
        \centering
        \includegraphics[width=\linewidth]{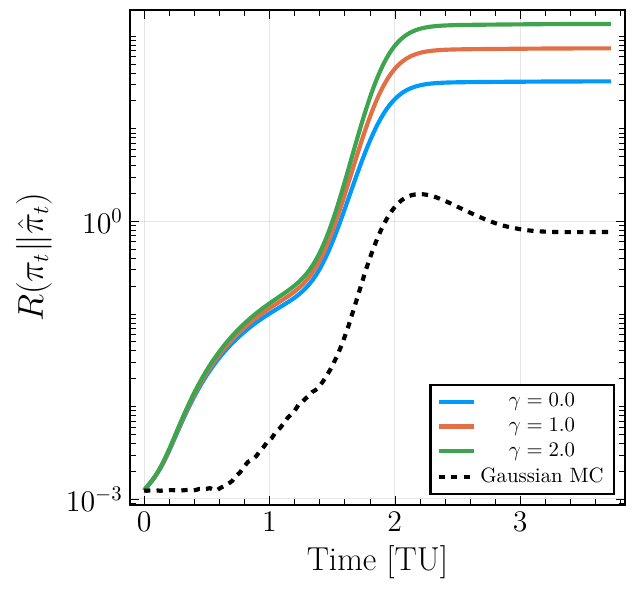} %
        \caption{}
    \end{subfigure}
    \caption{The true $R(\pi_t^{MC}\| \hat{\pi}_t)$ and its upper bound, derived in the present study, as a function of time for the scenario provided. The left- and right-hand sides correspond to $\sigma = 5 \times 10^{-6} \text{ km}/\text{s}^{3/2}$ and $\sigma = 1 \times 10^{-6} \text{ km}/\text{s}^{3/2}$, respectively.}
    \label{fig:kld}
\end{figure*}
Figure \ref{fig:orbits_and_MC} illustrates the scenario in the present study. The top panel shows the evolution of the $3-$sigma Gaussian reference covariance ellipse, pushed through the optimized feedback-controlled dynamics. This particular simulation uses a diffusion magnitude $\sigma = 1\times10^{-6} \text{ km}/\text{s}^{3/2}$. The black and blue lines are the nominal reference mean and the chaser spacecraft trajectory. Three different discrete times are emphasized in the bottom panels. A Monte Carlo simulation compares the nonlinear state dispersions with the expected linearized reference covariance evolution. Towards the end of the simulation history, the state dispersions in the Monte Carlo simulation become inconsistent with the reference distribution due to the accumulation of nonlinear errors.

Our relative entropy bound is assessed by comparing it with a truth-surrogate model. This model computes the \ac{kld} between a Gaussian reduction of the Monte Carlo history and our reference distribution. That is, we determine the empirical mean and covariance $\boldsymbol{\mu_k^{\text{MC}}} = \frac{1}{N^{\text{MC}}} \sum_{i = 1}^{N^{\text{MC}}} \boldsymbol{x}_k^i$ and $P_k^{\text{MC}} = \frac{1}{N^{\text{MC}} - 1}\sum_{i = 1}^{N^{\text{MC}}}\left(\boldsymbol{x}_k^i - \boldsymbol{\mu}_k^{\text{MC}}\right)\left(\boldsymbol{x}_k^i - \boldsymbol{\mu}_k^{\text{MC}}\right)^\top,$ where $N^{\text{MC}}$ is the number of Monte Carlo trials. The discrete time-varying relative entropy model is then
\begin{flalign}
    R_k(\pi_k^{MC}\| \hat{\pi}_k) &=\frac{1}{2}\bigg( \log\frac{|P_k|}{|P_k^{\text{MC}}|} + \operatorname{tr}\left(P_k^{-1} P_k^{\text{MC}}\right)\\
    &+ \left(\boldsymbol{\mu}_k - \boldsymbol{\mu}_k^{\text{MC}}\right)^\top P_k^{-1}\left(\boldsymbol{\mu}_k - \boldsymbol{\mu}_k^{\text{MC}} \right) \bigg), 
\end{flalign}
which is the exact \ac{kld} between two multivariate Gaussian distributions \cite{Zhang_2023}. To be precise, this model is not the true \ac{kld} because of the loss in information caused by the Gaussian reduction. However, for this study, it serves as a tractable surrogate quantity for comparison with our upper bound value. Despite this discrepancy, we will refer to this value as the \textit{true} value of the relative entropy for the sake of conciseness. 

We plot both the true \ac{kld} and the upper bound proposed in Figure \ref{fig:kld}. The left- and right-hand panels show this evolution for the $\sigma = 5\times10^{-6} \text{ km}/\text{s}^{3/2}$ and $\sigma = 1\times 10^{-6} \text{ km}/\text{s}^{3/2}$ cases, respectively. At the start of the trial, the Gaussian reference coincides closely with the true distribution, so the true relative entropy value is quite small. However, nonlinearity in the system dynamics progressively impacts the true distribution, especially towards the portion of the trajectory closest to perilune, significantly influencing the evolution of the relative entropy. Importantly, we see that our upper bound is indeed valid for all values of $\gamma$, and correlates well with the true relative entropy history. 

%
\begin{figure*}[htb!]
    \centering
    \begin{subfigure}{0.35\textwidth}
        \centering
        \includegraphics[width=\linewidth]{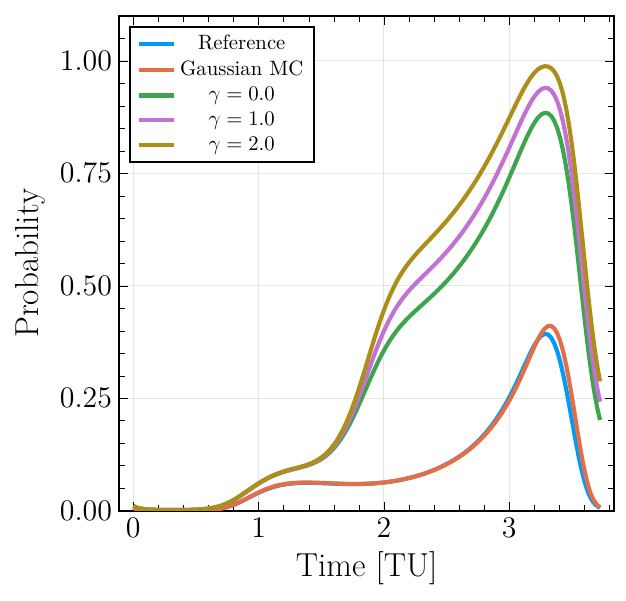} 
        \caption{}
    \end{subfigure}
    \hspace{1ex}
    \begin{subfigure}{0.35\textwidth}
        \centering
        \includegraphics[width=\linewidth]{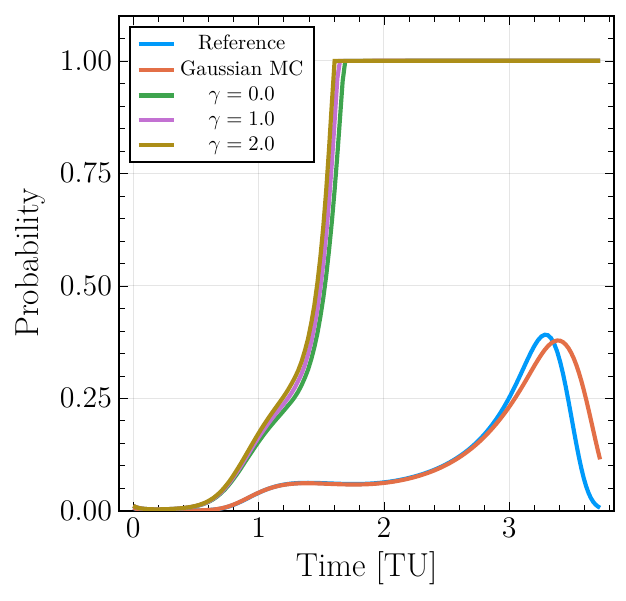} 
        \caption{}
    \end{subfigure}
    \caption{Risk probability evaluated under the nominal reference and true distributions, and the risk upper bound presented here. The left- and right-hand sides correspond to $\sigma = 5 \times 10^{-6} \text{ km}/\text{s}^{3/2}$ and $\sigma = 1 \times 10^{-6} \text{ km}/\text{s}^{3/2}$, respectively.}
    \label{fig:risk}
\end{figure*}

%
\begin{figure}[htb!]
    \centering
    \includegraphics[width=0.40\textwidth]{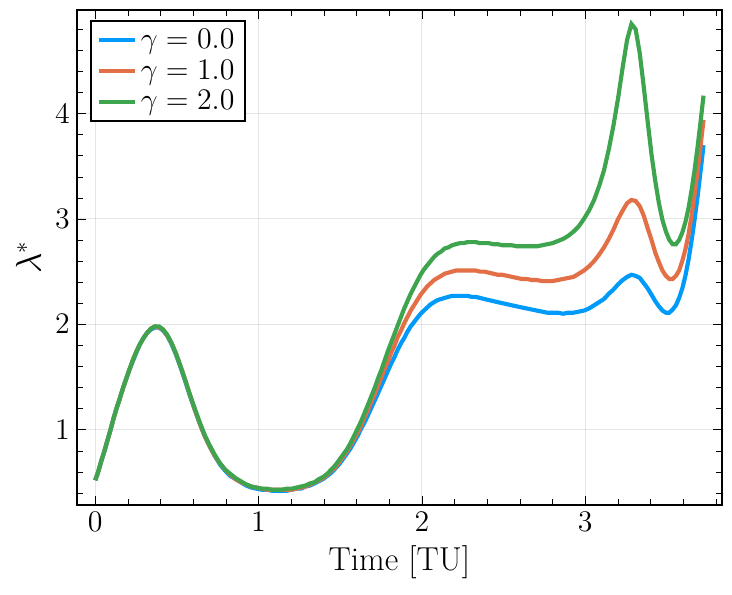}
    \caption{$\lambda^*$ as a function of time for the scenario using a diffusion magnitude $\sigma = 5\times 10^{-6} \text{ km}/\text{s}^{3/2}$.}\label{fig:optimal_lambda}
\end{figure}

In Figure \ref{fig:risk} we plot the risk probability evaluated under both the Gaussian reference and Monte Carlo Gaussian reduction as a function of time, along with the risk upper bound proposed. The left- and right-hand panels again correspond to $\sigma = 5\times10^{-6} \text{ km}/\text{s}^{3/2}$ and $\sigma = 1\times 10^{-6} \text{ km}/\text{s}^{3/2}$. Note the deviation in the risk probability values towards the latter half of the scenario for the true and reference distributions. Importantly, we find that the nominal risk underestimates the truth at the terminal stage of the spacecraft's trajectory. Our upper bound is indeed valid and constrains both the nominal and true risk -- despite its conservativeness. In Figure \ref{fig:optimal_lambda}, we show the optimal $\lambda^*$ for the case where $\sigma = 5\times10^{-6} \text{ km}/\text{s}^{3/2}$ used in generating the risk upper bound. The optimal value is time-varying, and we plot its evolution over the scenario. As mentioned previously, this optimal value depends on the value of the relative entropy upper bound, and we see this reflected in the plot.

\section{CONCLUSIONS}
In this work, we introduced an approach for evaluating risk in stochastic guidance optimization when there exists ambiguity in the state distribution. We achieved this by parameterizing a distribution ambiguity set by the relative entropy distance and determining a risk upper bound using a variational expression for exponential integrals. Additionally, we introduced a principled approach to determine an appropriate time-varying size of the ambiguity set, or relative entropy upper bound, which is particularly applicable for covariance steering applications in nonlinear dynamical systems. The latter bound is a function of a reference distribution's covariance, allowing one to control the ambiguity set in covariance steering applications. The current study formed the fundamentals for applying our methods to chance-constrained covariance steering, and future work aims to make these tools operational in convex optimization procedures.  

\bibliographystyle{IEEEtran}
\bibliography{biblio}

\addtolength{\textheight}{-12cm}   



\end{document}